\documentclass[a4paper,11pt,reqno]{amsart}

\edef\restoreparindent{\parindent=\the\parindent\relax}
\usepackage{parskip}
\usepackage{graphicx}
\usepackage{xcolor}
\usepackage{dsfont}
\usepackage{comment}
\restoreparindent
\numberwithin{equation}{section}
\usepackage[
    linktoc=page,       
    colorlinks=false,   
    pdfborder={0 0 1}   
]{hyperref}

\newtheorem{thm}{Theorem}
\newtheorem{lem}{Lemma}
\newtheorem{cor}{Corollary}
\newtheorem{conj}{Conjecture}
\newtheorem{Thm}{Theorem}

\newtheorem{Lem}{Lemma}

\theoremstyle{definition}
\newtheorem{rem}{Remark}
\newtheorem{defn}{Definition}
\newtheorem{example}{Example}
\DeclareMathOperator{\sgn}{sgn}

\newcommand{\be}{\begin{equation}}
\newcommand{\ee}{\end{equation}}

\newcommand{\blem}{\begin{lem}}
\newcommand{\elem}{\end{lem}}

\newcommand{\bdefn}{\begin{defn}}
\newcommand{\edefn}{\end{defn}}

\newcommand{\bthm}{\begin{thm}}
\newcommand{\ethm}{\end{thm}}

\newcommand{\bcor}{\begin{cor}}
\newcommand{\ecor}{\end{cor}}

\newcommand{\bconj}{\begin{conj}}
\newcommand{\econj}{\end{conj}}

\newcommand{\beg}{\begin{example}}
\newcommand{\eeg}{\end{example}}

\newcommand{\beqq}{\begin{eqnarray*}}
\newcommand{\eeqq}{\end{eqnarray*}}

\newcommand{\brem}{\begin{rem}}
\newcommand{\erem}{\end{rem}}

\newcommand{\bpf}{\begin{proof}}
\newcommand{\epf}{\end{proof}}

\begin{document} 

\bibliographystyle{amsplain}

\title[Schwarz-Pick Lemma for Invariant Harmonic Functions on the Complex Unit Ball]
{Schwarz-Pick Lemma for Invariant Harmonic Functions on the Complex Unit Ball}

\author{Kapil Jaglan}
\address{Kapil Jaglan, Department of Mathematics, Kyungpook National University, Daegu 41566, Republic of Korea.}
\email{kapiljaglan0501@gmail.com}

\author{Aeryeong Seo}
\address{Aeryeong Seo, Department of Mathematics, Kyungpook National University, Daegu 41566, Republic of Korea.}
\email{aeryeong.seo@knu.ac.kr}



\date{\today}

\begin{abstract}
This paper establishes a sharp Schwarz-Pick type inequality for real-valued invariant harmonic functions defined on the complex unit ball $\mathbb B^n$. The proof of this main result simultaneously provides a solution to a natural extension of the Khavinson conjecture for invariant harmonic functions, demonstrating that the sharp constants for the gradient and the radial derivative coincide. As further consequences of the main theorem, we derive two corollaries.
\end{abstract}

\maketitle
\tableofcontents
\pagestyle{myheadings} 
\markboth{}{}

\section{Introduction}
The well-known classical Schwarz-Pick lemma states that if $f$ is a holomorphic map from the unit disc into itself satisfying $f(0)=0$, then 
$$
|f'(z)| \leq \frac{1}{1-|z|^{2}} \quad \text{ for any } |z|<1.
$$
This inequality implies the fact that holomorphic self-maps of the unit disk are distance-decreasing with respect to the Poincar\'e metric and it shows that holomorphic maps are  controlled by the underlying hyperbolic metric.
The Schwarz–Pick lemma, together with its higher-dimensional generalizations to
holomorphic mappings between complex manifolds, plays a fundamental role in
several complex variables and complex geometry.
It underlies the theory of invariant metrics, such as the Kobayashi and Carath\'eodory
metrics, and provides a basic tool for understanding rigidity, hyperbolicity, and boundary behavior of holomorphic mappings.

This perspective naturally leads to the question of whether analogous Schwarz–Pick type inequalities hold for broader classes of mappings, in particular for (invariant) harmonic functions.
Although harmonic functions lack the strong structural properties of holomorphic functions, they nevertheless enjoy  rigidity phenomena under suitable conditions. In particular, one may ask whether harmonic functions defined on the unit disk or higher-dimensional domains satisfy sharp gradient estimates or distance-decreasing properties analogous to those in the classical Schwarz–Pick lemma.

Such Schwarz–Pick type results for harmonic functions have attracted considerable attention in recent years. They aim to understand how the hyperbolic geometry of the domain still governs the behavior of harmonic functions, even though these functions are not holomorphic.
For this line of research, see \cite{MR3619443, MR4490949,  MR2833528, MR3123573,MR3281685,MR3753181,MR3794078} and the references therein. In particular, Liu \cite{MR4490949} recently established a Schwarz–Pick type inequality for harmonic functions on the unit ball of $\mathbb{R}^n$.

\begin{Thm}[Liu \cite{MR4490949}]\label{ThmA}
Let $u$ be a real-valued bounded harmonic function on the unit ball $\mathbb{B}^n_{\mathbb{R}}=\{x\in \mathbb R^n : \|x\|<1\}$ of $\mathbb{R}^n$.
\begin{itemize}
    \item [(i)] When $n=2$ or $n\geq 4$, we have the inequality
    \be\label{eq7}
    |\nabla u(x)| \leq \frac{2\Gamma(\frac{n+2}{2})}{\sqrt{\pi}~\Gamma(\frac{n+1}{2})} \frac{1}{1-\|x\|^2}\sup_{y \in \mathbb{B}^n_{\mathbb{R}}}|u(y)|, \quad x \in \mathbb{B}^n_{\mathbb{R}}.
    \ee
Equality holds in \eqref{eq7} at one point $z$ if and only if $z = 0$ and $u = U \circ T$ for some orthogonal transformation $T$, where $U$ is the Poisson integral of the function that equals $1$ on a hemisphere and $-1$ on the remaining hemisphere. \\
\item [(ii)]  When $n=3$, we have
$$
|\nabla u(x)| < \frac{8}{3\sqrt{3}}\frac{1}{1-\|x\|^2}\sup_{y \in \mathbb{B}^3_{\mathbb{R}}}|u(y)|, \quad x \in \mathbb{B}^3_{\mathbb{R}}. 
$$
The constant $\frac{8}{3\sqrt{3}}$ here is the best possible.
\end{itemize}
\end{Thm}

This Liu’s work continues his earlier study \cite{MR4263696}, where he presented a complete proof of the Khavinson conjecture. The conjecture asserts that, for bounded harmonic functions on the unit ball of $\mathbb{R}^n$, the sharp constants governing the radial derivative and the gradient coincide. Gradual developments toward the solution of the Khavinson conjecture can be found in \cite{MR1165171, MR2839860, MR3725479, MR3619443, MR3975032}.

Let $\mathbb B^n = \{z\in\mathbb C^n: \|z\|<1\}$ be the complex unit ball in the complex Euclidean space $\mathbb C^n$ and let $K\colon \mathbb B^n\times\mathbb B^n\to\mathbb C$ be its (normalized) Bergman kernel given by 
$$
K(z,w) := \frac{1}{(1-\langle z,  w \rangle)^{n+1}}
$$
where $\langle , \rangle$ denote the Hermitian dot product, i.e. $\langle z,w\rangle = \sum_j z_j\overline w_j$.
Let $g = (g_{i\bar j})$ denote the Bergman metric defined by 
$$
g_{i\bar j} = (n+1)\left(\frac{\delta_{ij}}{1-\|z\|^2}
+\frac{\overline z_i z_j}{(1-\|z\|^2)^2}\right)
$$
and $g^{-1}=(g^{i\bar j})$ be its inverse, given by 
$$
g^{i\bar j} = \frac{1-\|z\|^2}{n+1}( \delta_{ij} - z_i\overline z_j).
$$
Denote by $\Delta_B$ the Laplace–Beltrami operator associated with the metric $g$.
A function $f \colon \mathbb B^n \to \mathbb C$ is called an invariant harmonic function if it satisfies $\Delta_B f \equiv 0$ on $\mathbb B^n$. Explicitly, $f$ satisfies
$$
\Delta _{B}(f)(z)= 4(1-\|z\|^2)\sum_{i,j=1}^{n}(\delta_{ij}-z_{i}\overline{z}_{j})\frac{\partial ^2 f}{\partial z_{i}\partial \overline{z}_{j}}(z),$$
where $\delta_{ij}$ is the Kronecker delta.
Let $\nabla$ denote the gradient with respect to the Euclidean metric of $\mathbb C^n$, i.e. 
for a real-valued function $f\colon\mathbb B^n\to\mathbb R$ and coordinate $z_j=x_j+\sqrt{-1}y_j$, 
$$
\nabla f = 4\,\text{Re }\sum_{j=1}^n \frac{\partial f}{\partial \overline z_j}\frac{\partial}{\partial  z_j} = \sum_{j=1}^n \left(
\frac{\partial f}{\partial x_j}\frac{\partial}{\partial x_j} + \frac{\partial f}{\partial y_j}\frac{\partial}{\partial y_j}\right).
$$

In this paper, we study invariant harmonic functions and generalize Theorem~\ref{ThmA} to the setting of invariant harmonic functions on $\mathbb B^n$.

\begin{thm}\label{thm1}
Let $h$ be a real-valued invariant harmonic function on the unit ball $\mathbb B^n \subset \mathbb C^n$ and $|h| \leq 1$ on $\mathbb B^n$. Then, 
\be\label{eq6}
\|\nabla h(z)\| \leq \frac{2\,\Gamma(n+1)}{\sqrt{\pi}~\Gamma(n+\frac{1}{2})} \frac{1}{1-\|z\|^2}, \quad z \in \mathbb B^n.
\ee 
Moreover, the inequality is sharp.
\end{thm}

In the specific case where $n=1$, inequality \eqref{eq6} reduces to:
$$
|\nabla h(z) | \leq \frac{4}{\pi} \frac{1}{1-|z|^2}, \quad z \in \mathbb{B}^1,
$$
which is consistent with the result established in \cite[Theorem 3]{MR1029679}.

Moreover, we show that the analogue of the Khavinson conjecture holds for bounded invariant harmonic functions on the complex unit ball (see Remark~\ref{remark}). Specifically, for any point in $\mathbb B^n$, the maximal gradient is attained when the derivative is taken in the radial direction.

As a consequence of Theorem~\ref{thm1}, we obtain two corollaries.
We denote by $\nabla_B$ the gradient with respect to $g$, i.e. $$
\nabla_B f =  
\sum_{i,j} g^{i\bar j} \frac{\partial f}{\partial \overline z_j}\frac{\partial}{\partial z_i},
$$
and by $\| v\|_B$ the norm of a vector $v$ with respect to $g$.

\begin{cor}\label{cor1}
Let $h$ be a real-valued invariant harmonic function on the unit ball $\mathbb B^n$ and $|h| \leq 1$ on $\mathbb B^n.$ Then, for any $z\in\mathbb B^n$, 
\be \nonumber
\frac{1-\|z\|^2}{2\sqrt{n+1}}\|\nabla h\|
\leq \|\nabla_B h(z)\|_B 
\leq \frac{2\,\Gamma(n+1)}{\sqrt{\pi(n+1)}\, \Gamma(n+\frac{1}{2})}
\ee
and the second inequality is sharp. As a consequence, we have
\begin{equation}\label{Lipschitz}
|h(z)-h(w)| \leq \frac{2\,\Gamma(n+1)}{\sqrt{\pi(n+1)} \,\Gamma(n+\frac{1}{2})}d_{\mathbb B^{n}}(z,w), 
\end{equation}
where $d_{\mathbb B^{n}}$ is the hyperbolic distance on $\mathbb B^{n}$.
\end{cor}

The next corollary shows that the inequality \eqref{eq6} in Theorem \ref{thm1} holds for vector-valued functions whose components are invariant harmonic functions. A similar result for vector-valued harmonic functions on the unit ball of $\mathbb R^n$ was obtained by the authors in \cite{MR4922257}.

\begin{cor}\label{cor2}
Let $H=(h_{1},....,h_{m}) \colon \mathbb B^n \rightarrow \mathbb R^m$ be a map whose components are invariant harmonic functions. Suppose that $\|H(z)\| \leq 1$ for all $z \in \mathbb B^n$. Then,
\be \label{eq17}
\|\nabla H(z)\| \leq \frac{2\,\Gamma(n+1)}{\sqrt{\pi}~\Gamma(n+\frac{1}{2})} \frac{1}{1-\|z\|^2}, \quad z \in \mathbb B^n,
\ee
where $\|\nabla H(z)\|$ denotes the operator norm of the Jacobian matrix $\nabla H(z) \in \mathbb R^{m \times 2n}$.
\end{cor}

\medskip
The organization of this paper is as follows. {In Section \ref{sec2}, we derive explicit formulas for the gradient of the Poisson-Szeg\"o kernel given in (\ref{eq18}) and for the transformation of the surface element on the unit ball $\mathbb{B}^n$. Section \ref{sec3} is devoted to the proofs of the main theorem and the corollaries. Finally, in Section \ref{sec4}, we provide an analogue of the classical Schwarz lemma to the setting of invariant harmonic functions, adapting techniques introduced by Burgeth \cite{MR1188585}.}

\section{Preliminaries}\label{sec2}
\subsection{Complex derivatives of the Poisson kernel} 
Let
\be\label{eq18}
\mathcal{P}_{z}(w)=\frac{1}{\sigma({\partial\mathbb B^n})} \frac{(1-\|z\|^2)^n}{|1- \langle z,w \rangle |^{2n}}
\ee
be the Poisson-Szeg\"o (or invariant Poisson) kernel of $\mathbb{B}^n$,
where $\sigma(\partial\mathbb B^n)$ denote the surface area of $\partial\mathbb B^n$.
Let $h(z)$ be a bounded invariant harmonic function on the unit ball $\mathbb B^n$. It has a radial boundary value 
$$
h^{*}(w)=\lim_{r \rightarrow 1^{-}}h(rw)
$$ 
for almost every $w \in \partial\mathbb B^n$. 
Moreover, $h(z)$ is represented as the Poisson integral 
\be\label{eq8}
h(z) = P[h^{*}](z)=\int_{{\partial\mathbb B^n}} \mathcal{P}_{z}(w)  h^{*}(w) d\sigma(w)
\ee
where $d\sigma$ is the standard surface area of $\partial\mathbb B^n$. For more details, see \cite[Theorem~4.3.3, Theorem 3.3.8]{rudin1980function} and \cite{MR1297545}.

Since $\mathcal{P}_{z}(w)$ is a real-valued function, the gradient is given as
$$
\nabla \mathcal{P}_z = 2 \sum_{j=1}^{n} \left(\frac{\partial \mathcal P_z}{\partial \overline z_{j}}\frac{\partial }{\partial z_{j}}+\frac{\partial \mathcal P_z}{\partial {z}_{j}}\frac{\partial}{\partial \overline{z}_{j}}\right)
= 4\,\text{Re }\sum_{j=1}^n \frac{\partial \mathcal P_z}{\partial \overline z_j}\frac{\partial}{\partial z_j}.
$$
 
The partial derivative $\partial {\mathcal{P}_z}/\partial \overline{z}_{j}$ is derived as follows. 
\begin{align}
       \frac{\partial \mathcal{P}_{z}(w)}{\partial \overline{z}_{j}}
  &  = \mathcal P_z(w)\frac{\partial }{\partial \overline{z}_{j}} \log \mathcal P_z(w) \nonumber \\
    &=\frac{1}{\sigma({\partial\mathbb B^n})} \frac{(1-\|z\|^2)^n}{|1- \langle z,w \rangle |^{2n}} \left(
    n\frac{\partial }{\partial \overline{z}_{j}} \log(1-\|z\|^2) 
    - n\frac{\partial }{\partial \overline{z}_{j}} \log(1-\langle w,z\rangle )
    \right) \nonumber \\
     &=\frac{1}{\sigma({\partial\mathbb B^n})}\frac{n(1-\|z\|^2)^n}{|1- \langle z,w \rangle |^{2n}} \left(
    \frac{-z_j}{1-\|z\|^2}
    - \frac{-w_j}{1-\langle w,z \rangle}
    \right). \label{eq5}
\end{align}
This gives
\be\label{eq1}
\frac{\partial \mathcal P_z}{\partial \overline z}
=\frac{n(1-\|z\|^2)^n}{\sigma({\partial\mathbb B^n})|1- \langle z,w \rangle |^{2n}} \left(
    \frac{-z}{1-\|z\|^2}
    + \frac{w}{1-\langle w,z \rangle}
    \right).
\ee
\subsection{Transformation formula for surface element}
For a fixed point $a \in \mathbb{B}^{n} \setminus \{0\}$, let $\phi_a$ be an automorphism of $\mathbb B^n$ given by 
\begin{equation}\label{involution}
\phi_{a}(z)=\frac{a-P_{a}(z)-s_{a}Q_{a}(z)}{1-\langle z,a \rangle}, \quad z\in \mathbb{B}^n,
\end{equation}
where $s_{a}=\sqrt{1-\|a\|^2}$, $P_{a}(z)=a\langle z, a \rangle / \|a\|^2$ represents the orthogonal projection of $\mathbb{C}^n$ onto the one dimensional subspace $[a]$ generated by $a$, and $Q_{a}(z)=z-P_{a}(z)$ is the orthogonal projection onto its orthogonal complement. In the special case $a =0$, we simply define $\phi_{0}(z)=-z$. 
Moreover, these maps are involutive in the sense that $\phi_{a} \circ \phi_{a} (z)=z$ and it satisfies $\phi_a(0)=0$.
The following lemmas are useful in the proof of theorem.

\begin{Lem}{\normalfont(\cite[Lemma 1.3]{MR2115155})}\label{LemmaA}
Suppose $a \in \mathbb{B}^n$. Then
$$
1 - \langle \phi_{a}(z), \phi_{a}(w) \rangle = \frac{(1-\langle a, a \rangle)(1-\langle z, w \rangle)}{(1-\langle z, a \rangle)(1-\langle a, w \rangle)}
$$
for all $z$ and $w$ on the closed unit ball $\overline{\mathbb{B}^n}$.
\end{Lem}

\begin{lem}\label{Lemma1}
For any $\gamma\in \text{Aut}(\mathbb B^n)$, we have 
    $$
\gamma^*d\sigma (\eta) = \left(\frac{1-\|z\|^2}{|1-\langle z,\eta \rangle|^2}\right)^{n}d\sigma (\eta)
$$
for $z = \gamma^{-1}(0)$.
\end{lem}
\begin{proof}
    By \cite[Lemma 4.1]{S25}, we have $$
    \gamma^* d\sigma (\eta)
    =|\det \mathcal J_{\mathbb C}\gamma(\eta)|^{\frac{2n}{n+1}} d\sigma(\eta)
    $$
    for any $\eta \in \partial \mathbb B^n$ and $\gamma \in \operatorname{Aut}(\mathbb B^n)$.
    On the other hand, Lemma~1.7 in \cite{MR2115155} gives
    $$
    \det\mathcal J_{\mathbb R}\gamma(\eta) 
    =|\det \mathcal J_{\mathbb C}\gamma(\eta)|^2
    = \left( \frac{1-|z|^2}{|1-\langle \eta, z \rangle |^2}
    \right)^{n+1}.
    $$
    Combining these two results yields the desired lemma.
\end{proof}

\section{Proof of the Main Result and Related Results}\label{sec3}

\subsection{Proof of the inequality in Theorem~\ref{thm1}}

For $z \in \mathbb{B}^n$ and a direction $l=(l_1,\ldots, l_n) \in \partial \mathbb{B}^n$, the Poisson integral representation \eqref{eq8} of $h$ gives
$$
\left\langle \frac{\partial h}{\partial \overline z},l\right\rangle = \int_{{\partial\mathbb B^n}} \left\langle \frac{\partial \mathcal P_z(w)}{\partial \overline z},l \right\rangle  h^{*}(w) d\sigma (w).
$$
Viewing $l$ as a real vector with components
$(\text{Re }l_1, \text{Im }l_1,\ldots, \text{Re }l_n, \text{Im }l_n)$, the dot product $\nabla \mathcal P_z \cdot l$ is given by 
$$
\nabla \mathcal P_z \cdot l
=
\langle \nabla \mathcal{P}_{z}(w),l \rangle=\sum_{j=1}^{n} \left(\frac{\partial \mathcal{P}}{\partial z_{j}} l_{j}+\frac{\partial \mathcal{P}}{\partial \overline{z}_{j}} \overline{l_{j}}\right)
= 2\,\text{Re }\left\langle \frac{\partial \mathcal P_z}{\partial \overline z},l\right\rangle.
$$

Let us first evaluate 
\be \label{eq14}
\mathcal{C}(z,l)
=\int_{\partial\mathbb B^n} \left|\nabla\mathcal P_z \cdot l\right| d\sigma(w)
= \int_{\partial \mathbb{B}^n} 2\left|\text{Re}\left\langle \frac{\partial \mathcal P_z(w)}{\partial \overline z},l \right\rangle \right| d\sigma (w).
\ee
We make the change of variable $w=\phi_{z}(\eta)$, where $\phi_{z}$ is the involutive automorphism given in \eqref{involution}.
Lemma~\ref{Lemma1} gives the boundary Jacobian formula 
$$
d\sigma (w) = \left(\frac{1-\|z\|^2}{|1-\langle z,\eta \rangle|^2}\right)^{n}d\sigma (\eta). 
$$
This implies  
$$
\mathcal{C}(z,l)
= \int_{\partial \mathbb{B}^n} 2\left|\text{Re} \left\langle \frac{\partial \mathcal P_z}{\partial \overline z}(\phi_{z}(\eta)),l \right\rangle \right| \left(\frac{1-\|z\|^2}{|1-\langle z,\eta \rangle|^2}\right)^{n}d\sigma (\eta).
$$

The replacements $a \rightarrow z$, $z \rightarrow 0$ and $w \rightarrow \eta$ in Lemma~\ref{LemmaA} give $$1-\langle z,w \rangle = \frac{1-\|z\|^2}{1-\langle z,\eta \rangle}.$$ Equation \eqref{eq5} then implies that
$$
\frac{\partial \mathcal{P}_z}{\partial \overline{z}_{j}}=\frac{n}{\sigma({\partial\mathbb B^n})(1-\|z\|^2)^{n+1}}\left(\frac{-z+\phi_{z}(\eta)(1-\overline{\langle z,\eta \rangle})}{|1-\langle z,\eta \rangle|^{-2n}} \right).
$$
Therefore
$$
\mathcal{C}(z,l)=\frac{2n}{\sigma({\partial\mathbb B^n})(1-\|z\|^2)}\int_{\partial \mathbb{B}^n} |\text{Re}{((1-\overline{\langle z,\eta \rangle}) \langle \phi_{z}(\eta), l \rangle -\langle z,l \rangle )}| d\sigma (\eta)
$$
and further simplification  
\begin{align*}
(1-\overline{\langle z,\eta \rangle}) \langle \phi_{z}(\eta), l \rangle -\langle z,l \rangle  & = (1-\overline{\langle z,\eta \rangle}) \left \langle \frac{z-P_{z}(\eta)-s_{z}Q_{z}(\eta)}{1-\langle \eta,z \rangle}, l \right \rangle -\langle z,l \rangle \\   
& =  \langle z-P_{z}(\eta)-s_{z}Q_{z}(\eta), l  \rangle -\langle z,l \rangle \\
& = - \langle P_{z}(\eta) + s_{z}Q_{z}(\eta), l  \rangle
\end{align*}
gives 
\begin{align}
\mathcal{C}(z,l) & =\frac{2n}{\sigma({\partial\mathbb B^n})(1-\|z\|^2)}\int_{\partial \mathbb{B}^n} |\text{Re}{(\langle P_{z}(\eta) + s_{z}Q_{z}(\eta), l  \rangle )}| d\sigma (\eta) \notag \\
& = \frac{2n}{\sigma({\partial\mathbb B^n})(1-\|z\|^2)}\int_{\partial \mathbb{B}^n} \left|\text{Re}{\left (\frac{\langle \eta,z \rangle \langle z,l \rangle}{\|z\|^2} + s_{z}\left(\langle \eta,l \rangle  - \frac{\langle \eta,z \rangle \langle z,l \rangle}{\|z\|^2} \right)\right)} \right| d\sigma (\eta) \notag \\
& = \frac{2n}{\sigma({\partial\mathbb B^n})(1-\|z\|^2)}\int_{\partial \mathbb{B}^n} |\text{Re}{\langle \eta, v(z,l) \rangle}| d\sigma (\eta) \label{eq9}
\end{align}
where $$
v(z,l):=\begin{cases}
s_{z}l~+~(1-s_{z})\langle l,z \rangle \frac{z}{\|z\|^2}, &\text{ if } z\neq 0\\
l,& \text{ if } z=0.
\end{cases}
$$
The spherical coordinates $\eta_{1}=e^{i\psi_{1}}\cos\theta_{1}$, $\eta_{2}=e^{i\psi_{2}}\sin\theta_{1}\cos\theta_{2}$,$\cdots $ and by rotational invariance of $d\sigma$, assuming $v=(\|v\|,0,0,...,0)$, we obtain
\begin{align}
\mathcal{C}(z,l) & = \frac{2n \|v\|}{\sigma({\partial\mathbb B^n})(1-\|z\|^2)}\int_{\partial \mathbb{B}^n} |\text{Re}({\eta_{1}})|d\sigma (\eta) \notag \\
& =\frac{2n \|v\|}{\sigma({\partial\mathbb B^n})(1-\|z\|^2)}\int_{0}^{\frac{\pi}{2}}\int_{0}^{2\pi}\int_{\mathbb{S}^{2n-3}}|\cos \psi_{1} \cos \theta_{1}|\cos \theta_{1} \sin^{2n-3}\theta_{1}d\sigma_{\mathbb{S}^{2n-3}} (\eta)d\psi_{1}d\theta_{1} \notag \\
& = \frac{2n \|v\|\sigma(\mathbb{S}^{2n-3})}{\sigma({\partial\mathbb B^n})(1-\|z\|^2)}\int_{0}^{2\pi}|\cos \psi_{1}| d\psi_{1} \int_{0}^{\frac{\pi}{2}} \cos^{2} \theta_{1} \sin^{2n-3}\theta_{1}d\theta_{1} \notag \\
& = \frac{2n \|v\|\sigma(\mathbb{S}^{2n-3})}{\sigma({\partial\mathbb B^n})(1-\|z\|^2)}  \frac{\sqrt{\pi}~\Gamma(n-1)}{\Gamma(n+\frac{1}{2})} \notag \\
& = \frac{2n \|v\|}{1-\|z\|^2}  \frac{\Gamma(n)}{\sqrt{\pi}~\Gamma(n+\frac{1}{2})}. \label{eq12}
\end{align}
This gives that 
$$
|\nabla h(z)\cdot l  | \leq \frac{2n \|v\|}{1-\|z\|^2} \frac{\Gamma(n)}{\sqrt{\pi}~\Gamma(n+\frac{1}{2})}.
$$
Further simplification gives that
\begin{align}
\|v\|^2 & = s^{2}_{z}\|l\|^{2} + (1-s_{z})^{2}\frac{|\langle l,z \rangle|^{2}}{\|z\|^2}+2s_{z}(1-s_{z})\text{Re}{\left \langle l, \frac{\langle l,z \rangle z}{\|z\|^{2}} \right \rangle} \notag \\
& = 1-\|z\|^{2}+ |\langle l, z \rangle|^{2} \quad \text{(since $s_{z}=\sqrt{1-\|z\|^{2}}$)}. \label{eq13}
\end{align}
Clearly $\|v\|=\sqrt{1-\|z\|^{2}+ |\langle l, z \rangle|^{2}} \leq 1$, and hence 
$$
|\nabla h(z)| = \sup_{l \in \partial\mathbb B^n} | \langle \nabla h(z),l \rangle| \leq \frac{2n}{1-\|z\|^2}  \frac{\Gamma(n)}{\sqrt{\pi}~\Gamma(n+\frac{1}{2})}.
$$ 
This completes the proof of the inequality. 

\subsection{Sharpness of the inequality}\label{sharpness}
For a fixed $z \in \mathbb{B}^n$, recall from equations \eqref{eq14} and \eqref{eq12} that
$$
\mathcal{C}(z,l)
= \int_{\partial \mathbb{B}^n} | \nabla \mathcal{P}_{z}(w) \cdot l  | d\sigma (w)=\frac{2n \|v\|}{1-\|z\|^2}  \frac{\Gamma(n)}{\sqrt{\pi}~\Gamma(n+\frac{1}{2})}.
$$
Consider the inequality
\be \label{eq11}
|\nabla h(z)\cdot l  | 
= \left| \int_{{\partial\mathbb B^n}}  (\nabla \mathcal{P}_{z}(w)\cdot l) \,h^{*}(w) d\sigma (w) \right|
\leq \int_{{\partial\mathbb B^n}} \left| \nabla \mathcal{P}_{z}(w) \cdot l \right|  |h^{*}(w)| d\sigma (w).
\ee

Remark that since
$$
 \sgn(\nabla \mathcal{P}_{z}(w) \cdot l )(\nabla \mathcal{P}_{z}(w) \cdot l)  =| \nabla \mathcal{P}_{z}(w)\cdot l |
$$ 
whenever $\nabla \mathcal P_z(w)\cdot l\neq 0$, we obtain the equality in \eqref{eq11} for the boundary function 
$$
h^{*}(w)= \sgn( \nabla \mathcal{P}_{z}(w)\cdot l ) \quad \text{ a.e }~ w\in \partial\mathbb B^n
$$ or 
$$
h^{*}(\phi_{z}(\eta))=\sgn(\text{Re}{\langle \eta, v(z,l) \rangle})\quad  \text{ a.e. }~w\in \partial\mathbb B^n,
$$ 
using equation \eqref{eq9}. In other words
\begin{equation*}
h^{*}(\phi_{z}(\eta))=
\begin{cases}
1,  & \text{Re}{\langle \eta, v(z,l) \rangle} > 0, \vspace{0.15cm} \\
-1, & \text{Re}{\langle \eta, v(z,l) \rangle} < 0,
\end{cases} \hspace{0.5cm} \text{a.e.}
\end{equation*} 

For $z\in \mathbb B^n$ and $l\in \partial\mathbb B^n$, define the hemisphere 
$$
H_{z,l}=\{\eta \in \partial\mathbb B^n \colon \text{Re}{\langle \eta, v(z,l) \rangle} > 0 \},
$$
and a boundary function by
\begin{equation*}
\psi_{z,l}(\eta)=
\begin{cases}
1,  & \text{ if } \eta \in H_{z,l}, \\
-1, & \text{ if } \eta \notin H_{z,l}.
\end{cases} \hspace{0.5cm} 
\end{equation*}
Now define a new boundary function by pulling it back to the original boundary variable $w$ using the involution $\eta = \phi_{z}(w)$ (since $\phi_{z} \circ \phi_{z}=I$):
$$ 
h^{*}_{z,l}(w) := \psi_{z,l}(\phi_{z}(w))= \sgn( \text{Re}{\langle \phi_{z}(w), v(z,l) \rangle}).
$$
Let $h_{z,l}$ be the Poisson integral representation of $h^{*}_{z,l}$ given as
$$
h_{z,l}(\xi) = \int_{{\partial\mathbb B^n}} \mathcal{P}_{\xi}(w) \,h^{*}_{z,l}(w) d\sigma(w).
$$
This gives that $h_{z,l}$ is an invariant harmonic function on $\mathbb B^n$ and $|h_{z,l}(\xi)| \leq 1$ for all $\xi \in \mathbb B^n$. 

{\bf Claim:}
The function $h_{z,l}$ defined as above gives equality in \eqref{eq6}, Theorem \ref{thm1} and hence is the extremal function. 

The proof of the claim is as follows: Since $h^{*}_{z,l}(w)= \sgn( \text{Re}{\langle \eta, v(z,l) \rangle}) $ and pointwise a.e.
$$
\quad \text{Re}{\langle \eta, v(z,l) \rangle} ~\sgn( \text{Re}{\langle \eta, v(z,l) \rangle}) = |\text{Re}{\langle \eta, v(z,l) \rangle}|,
$$
by repeating the computations as done in the proof of Theorem \ref{thm1}, we can show that  
$$
 \nabla h_{z,l}(z)\cdot l 
 = \int_{\partial \mathbb{B}^n} | \nabla \mathcal{P}_{z}(w)\cdot l | d\sigma (w)
=\mathcal{C}(z,l) 
$$
and hence $| \nabla h_{z,l}(z)\cdot l|
=\mathcal{C}(z,l).$ In equation \eqref{eq13}, whenever $l$ is parallel to $z$, i.e., $l=\hat{z}:=z/\|z\|$ for $z \neq 0$, and choosing $l=e_{1}$ for $z=0$, we have $\|v(z,l)\|=1$. Therefore, for the function $h_{z,\hat{z}}$, using equation \eqref{eq12}, we have
$$
|\nabla h_{z,\hat{z}}(z) | 
\geq | \nabla h_{z,\hat{z}}(z)\cdot \hat{z}|
=\mathcal{C}(z,\hat{z})=\frac{2n}{1-\|z\|^2}  \frac{\Gamma(n)}{\sqrt{\pi}~\Gamma(n+\frac{1}{2})},
$$ 
and hence the inequality is sharp at any arbitrary point $z \in \mathbb{B}^n$. 

\brem\label{remark}
Equations \eqref{eq12} and \eqref{eq13} show that the direction parallel to $z$ (the radial direction) gives the maximum size of $\mathcal{C}(z,l)$ in the estimate $|\langle \nabla h(z),l \rangle | \leq \mathcal{C}(z,l)$ for the directional derivative of bounded invariant harmonic function $h$ defined on the unit ball $\mathbb B^n$. The sharpness of the inequality \eqref{eq6} then shows that the same constant works in the estimate for the gradient of function $h.$ This gives a solution to the analogous Khavinson conjecture for bounded invariant harmonic functions on the unit ball $\mathbb B^n$.
\erem

\subsection{Proof of corollaries}
\bpf[Proof of Corollary \ref{cor1}] 
Since \begin{equation}\nonumber
\begin{aligned}
    \|\nabla_B h(z)\|_B^2 
    &= \|\nabla_Bh\circ\phi_z(0)\|_B^2
    = \|\nabla_B(h\circ\phi_z)(0)\|^2_B
    = \frac{1}{n+1}\| \nabla (h\circ\phi_z)(0)\|^2\\
    &\leq \frac{1}{n+1}\left(\frac{2\,\Gamma(n+1)}{\sqrt{\pi}~\Gamma(n+\frac{1}{2})}\right)^2,
\end{aligned}
\end{equation}
and 
\begin{equation}\nonumber
\begin{aligned}
    \|\nabla_B h(z)\|_B^2 
    &= \frac{1-\|z\|^2}{n+1}\left( \sum_j\left| \frac{\partial h}{\partial z_j}\right|^2 - \left|\sum_j z_j\frac{\partial h}{\partial z_j}\right|^2 \right)
    \geq \frac{(1-\|z \|^2)^2}{4(n+1)}\|\nabla h\|^2,
    \end{aligned}
    \end{equation}
    the inequalities are proved. Since \eqref{eq6} is sharp at $z=0$, we also obtain the sharpness of the second inequality.
    
    Let $\gamma:[0,1] \rightarrow \mathbb B^n$ be a piecewise $C^1$ curve with $\gamma(0)=z$ and $\gamma(1)=w$. The chain rule and the Cauchy-Schwarz inequality then give
$$
|h(z)-h(w)| \leq \left| \int_{0}^{1} \langle \nabla_{B} h (\gamma(t)), \gamma'(t)\rangle_{B}\, dt  \right| \leq \int_{0}^{1} \|\nabla_{B} h (\gamma(t))\|_{B} \| \gamma'(t)\|_{B}\, dt.
$$
Using the inequality in the corollary, we obtain
$$
|h(z)-h(w)| 
\leq \frac{2\,\Gamma(n+1)}{\sqrt{\pi(n+1)}~\Gamma(n+\frac{1}{2})} \int_{0}^{1} \| \gamma'(t)\|_{B}\, dt .
$$
By taking the infimum  over all such curves $\gamma$ joining $z$ and $w$, we obtain \eqref{Lipschitz}.
\epf

\bpf[Proof of Corollary \ref{cor2}]
Let $v \in \mathbb R^m$ be any unit vector, $\|v\|=1$, and define the scalar function $G_{v}(z)=  H(z)\cdot v $. Clearly, the function $G_{v}(z)$ is a real-valued invariant harmonic function on $\mathbb B^n$ and the Cauchy-Schwarz inequality $|G_{v}(z)| \leq \|H(z)\| \|v\| \leq 1$ for all $z \in \mathbb B^n$, further gives that $G_{v}(z)$ satisfies the assumptions of Theorem~\ref{thm1}. Inequality \eqref{eq6}) then gives
$$
\|\nabla G_{v}(z)\| = |(\nabla H(z))^{T}v| \leq \frac{2\,\Gamma(n+1)}{\sqrt{\pi}~\Gamma(n+\frac{1}{2})} \frac{1}{1-\|z\|^2}.
$$
Since $v$ is an arbitrary unit vector in $\mathbb R^m$, taking the supremum over all such $v$ yields the desired inequality \eqref{eq17}. 
\epf

\section{Schwarz lemma for invariant harmonic functions using Burgeth's method}\label{sec4} 

The classical Schwarz lemma states that the inequalities 
$$
|f(z)| \leq |z| \quad \text{for all} \quad |z|<1, \quad \text{and} \quad |f'(0)|\leq 1 
$$
hold for all holomorphic mappings $f$ from the unit disk into itself such that $f(0)=0$.
In this section, we adapt the approach of Burgeth \cite{MR1188585} to obtain an analogous Schwarz type result for invariant harmonic functions. Let $\mathds{1}_{A}$ stands for the characteristic function of a set $A \subset {\partial\mathbb B^n}$. Set
\be \label{eq15}
M_{c}^{n}(\|z\|)=2\int_{{\partial\mathbb B^n}}\mathds{1}_{S(c,\hat{z})} ~\mathcal{P}_{z}d\sigma -1,
\ee 
where $z \in \mathbb B^n$ and $$
S(c,\hat{z})=\{w \in {\partial\mathbb B^n}: \text{Re}{\langle \hat{z},w\rangle}>\cos(\alpha(c)),~\alpha(c)\in[0,\pi]\}
$$ 
is a spherical cap with polar angle $\alpha(c)$, center at $\hat{z}$, and of measure $c$.

\begin{thm}\label{thm2}
Let $h$ be a real-valued invariant harmonic function on the unit ball $\mathbb B^n$ such that $|h| \leq 1$, and $h(0)=a,$ $-1<a<1$. Then, for $c=(a+1)/2$ and for all $z \in \mathbb B^n$
$$
h(z) \leq M_{c}^{n}(\|z\|).
$$
 Equality holds if and only if $h$ is the Poisson integral of the characteristic function of a spherical cap on $\partial \mathbb B^n$, up to a rotation. 
\end{thm}
\bpf
The proof follows the same lines as that of \cite[Theorem 1]{MR1188585} and is therefore omitted. 
\epf

In the planar case $n=1$, the extremal function $M_{c}^{1}$ can be computed directly from \eqref{eq15} on $\mathbb B^{1}$ yielding, 
\be \label{extremal1}
M_{c}^{1}(\|z\|)=\frac{4}{\pi}\arctan\left(\frac{1+\|z\|}{1-\|z\|}\tan \frac{\alpha(c)}{2} \right)-1.
\ee
In particular, for $a=0$ (so $\alpha (c)=\pi/2$) one obtains $M_{1/2}^{1}(\|z\|)=\frac{4}{\pi}\arctan (\|z\|)$, recovering the classical estimate of Schwarz \cite{schwarz1972}.

Let us simplify the integral on the right-hand side of \eqref{eq15} using rotational invariance of $\sigma$ and rewriting it in spherical coordinates to obtain an explicit integral representation depending only on $r=\|z\|$. This makes the extremal function look transparent and recovers the relation \eqref{extremal1}. Assume $n \geq 2$. We use the identification $\mathbb{C}^{n} \cong \mathbb{R}^{2n}$ and introduce the real spherical coordinates on $\partial\mathbb B^n \cong \mathbb{S}^{2n-1}$. Write a point $w \in \partial\mathbb B^n \subset{\mathbb{C}^{n}}$ as $w=(w_{1}, w_{2},...,w_{n})$, such that $w_{j}=a_{j}+ib_{j}$, and the corresponding real vector is $(a_{1},b_{1},....,a_{n},b_{n}) \in \mathbb{R}^{2n}$.  Introducing the real spherical coordinates $a_{1}=\cos\theta_{1}$, $b_{1}=\sin\theta_{1}\cos\theta_{2}$,$\cdot \cdot \cdot $ and by rotational invariance of $\sigma$, assuming $z=(\|z\|,0,0,...,0)$, we derive
\begin{align*}
&\quad M_{c}^{n}(\|z\|)\\
&=\frac{2}{\sigma({\partial\mathbb B^n})}\int_{0}^{\alpha(c)}\int_{0}^{\pi}\int_{\mathbb{S}^{2n-3}}\frac{(1-\|z\|^2)^{n}}{|1-\|z\|\overline{w_{1}}|^{2n}}\sin^{2n-2}\theta_{1} \sin^{2n-3}\theta_{2}d\sigma_{\mathbb{S}^{2n-3}}(\zeta)d\theta_{2}d\theta_{1} -1 \\
& = \frac{2(1-\|z\|^2)^{n}\sigma(\mathbb{S}^{2n-3})}{\sigma({\partial\mathbb B^n})}\int_{0}^{\alpha(c)}\int_{0}^{\pi}\frac{\sin^{2n-2}\theta_{1} \sin^{2n-3}\theta_{2}}{[(1-\|z\|\cos \theta_{1})^{2}+(\|z\|\sin \theta_{1} \cos \theta_{2})^{2}]^{n}} d\theta_{2}d\theta_{1} -1 \\
& = \frac{2~\Gamma (n)(1-\|z\|^2)^{n}}{\pi \Gamma (n-1)}\int_{0}^{\alpha(c)}\int_{0}^{\pi}\frac{\sin^{2n-2}\theta_{1} \sin^{2n-3}\theta_{2}}{[(1-\|z\|\cos \theta_{1})^{2}+(\|z\|\sin \theta_{1} \cos \theta_{2})^{2}]^{n}} d\theta_{2}d\theta_{1} -1,
\end{align*}
using the fact that $\sigma(\mathbb{S}^{k-1})=\frac{2\pi ^{k/2}}{\Gamma(k/2)}$ for any $k\in\mathbb N$.

\medskip
\subsection*{Acknowledgement}
This work was supported by the National Research Foundation of Korea (NRF) grant funded by the Korea government (MSIT) (No. RS-2025-00561084).


\begin{thebibliography}{150}

\bibitem{MR1188585} Burgeth, B.,
\emph{A {S}chwarz lemma for harmonic and hyperbolic-harmonic functions in higher dimensions},
Manuscripta Math. {\bf 77}(2-3) (1992), 283--291.

\bibitem{MR1029679} Colonna, F.,
\emph{The {B}loch constant of bounded harmonic mappings},
Indiana Univ. Math. J. {\bf 38}(4) (1989), 829--840.

\bibitem{MR3123573} Chen, H.,
\emph{The {S}chwarz-{P}ick lemma and {J}ulia lemma for real planar harmonic mappings},
Sci. China Math. {\bf 56}(11) (2013), 2327--2334.

\bibitem{MR1165171} Khavinson, D.,
\emph{An extremal problem for harmonic functions in the ball},
Canad. Math. Bull. {\bf 35}(2) (1992), 218--220.

\bibitem{MR2839860} Kresin, G. and Maz'ya, V.,
\emph{Sharp pointwise estimates for directional derivatives of harmonic functions in a multidimensional ball},
J. Math. Sci. {\bf 169}(2) (2010), 167--187.

\bibitem{MR3725479} Kalaj, D.,
\emph{A proof of {K}havinson's conjecture in {$\mathbb{R}^{4}$}},
Bull. Lond. Math. Soc. {\bf 49}(4) (2017), 561--570.

\bibitem{MR2833528} Kalaj, D. and Vuorinen, M.,
\emph{On harmonic functions and the {S}chwarz lemma},
Proc. Amer. Math. Soc. {\bf 140}(1) (2012), 161--165.

\bibitem{MR4263696} Liu, C.,
\emph{A proof of the {K}havinson conjecture},
Math. Ann. {\bf 380}(1-2) (2021), 719--732.

\bibitem{MR4490949} Liu, C.,
\emph{Schwarz-pick lemma for harmonic functions},
Int. Math. Res. Not. IMRN {\bf 19} (2022), 15092--15110.

\bibitem{MR3281685} Markovi\'{c}, M.,
\emph{On harmonic functions and the hyperbolic metric},
Indag. Math. {\bf 26}(1) (2015), 19--23.

\bibitem{MR3619443} Markovi\'{c}, M.,
\emph{Solution to the {K}havinson problem near the boundary of the unit ball},
Constr. Approx. {\bf 45}(2) (2017), 243--271.

\bibitem{MR3753181} Melentijevi\'{c}, P.,
\emph{Invariant gradient in refinements of {S}chwarz and {H}arnack inequalities},
Ann. Acad. Sci. Fenn. Math. {\bf 43}(1) (2018), 391--399.

\bibitem{MR3794078} Mateljevi\'{c}, M.,
\emph{Schwarz lemma and {K}obayashi metrics for harmonic and holomorphic functions},
J. Math. Anal. Appl. {\bf 464}(1) (2018), 78--100.

\bibitem{MR3975032} Melentijevi\'{c}, P.,
\emph{A proof of the {K}havinson conjecture in {$\mathbb{R}^{3}$}},
Adv. Math. {\bf 352} (2019), 1044--1065.

\bibitem{rudin1980function} Rudin, W.,
\emph{Function Theory in the Unit Ball of     $\mathbb{C}^{n}$},
Grundlehren Math. Wiss., Springer, (1980).

\bibitem{schwarz1972} Schwarz, H. A.,
\emph{Gesammelte mathematische abhandlungen},
American Mathematical Soc., {\bf 260} (1972).

\bibitem{S25} Seo, A.,
\emph{On discrete subgroups of the complex unit ball},
Math. Nachr. {\bf 298}(10) (2025), 3272--3286.

\bibitem{MR1297545} Stoll, M.,
\emph{Invariant potential theory in the unit ball of {${\bf C}^n$}},
London Mathematical Society Lecture Note Series, Cambridge University Press, {\bf 199} (1994).

\bibitem{MR4922257} Xu, Z. and Yu, T. and Huo, Q.,
\emph{Schwarz lemma for harmonic functions in the unit ball},
Proc. Edinb. Math. Soc. (2) {\bf 68}(2) (2025), 616--633.

\bibitem{MR2115155} Zhu, K.,
\emph{Spaces of holomorphic functions in the unit ball},
Graduate Texts in Mathematics, Springer-Verlag, New York {\bf 226} (2005).

\end{thebibliography}
\end{document}